\documentclass[10pt,english]{article}
\usepackage[T1]{fontenc}
\usepackage[latin9]{inputenc}
\usepackage{amsthm}
\usepackage{amsmath}
\usepackage{amssymb}
\usepackage{amsfonts}
\usepackage{mathrsfs}
\usepackage{babel}
\usepackage{hyperref}
\usepackage{fullpage}

\numberwithin{equation}{section}
\numberwithin{figure}{section}

\theoremstyle{plain}
\newtheorem{thm}{Theorem}

\theoremstyle{definition}
\newtheorem{defn}{Definition}

\theoremstyle{remark}
\newtheorem{rem}{Remark}

\theoremstyle{plain}
\newtheorem{prop}{Proposition}

\theoremstyle{plain}
\newtheorem{lem}{Lemma}

\begin{document}

\title{A percolation process on the binary tree where large finite clusters are frozen}

\author{Jacob van den Berg\thanks{CWI and VU University, Amsterdam; J.van.den.Berg@cwi.nl}, 
Demeter Kiss\thanks{CWI, research supported by NWO; D.Kiss@cwi.nl} and Pierre Nolin\thanks{Courant Institute, NYU, New York; nolin@cims.nyu.edu}}

\maketitle
\begin{abstract}
We study a percolation process on the planted binary tree,
where clusters freeze as soon as they become larger than some fixed
parameter $N.$ We show that as $N$ goes to infinity, the process
converges in some sense to the frozen percolation process introduced
by Aldous in \cite{Aldous2000}.

In particular, our results show that the asymptotic behaviour differs substantially from that on the square lattice, on which a similar process has been studied recently by van den Berg, de Lima and Nolin \cite{Berg}.
\end{abstract}


\section{Introduction and statement of results}

Aldous \cite{Aldous2000} introduced a percolation
process  where clusters are frozen when they get infinite, which can be described as follows. Let $G=\left(V,E\right)$ be an arbitrary simple graph with vertex
set $V,$ and edge set $E.$ On every edge $e\in E,$ there is a clock
which rings at a random time $\tau_{e}$ with uniform distribution
on $\left[0,1\right],$ these random times $\tau_{e}, \, e\in E,$
being independent of each other. At time $0,$ all the edges are closed, and then each edge $e=(u,v)\in E$ becomes open at time $\tau_{e}$
if the open clusters of $u$ and $v$ at that time are both finite
-- otherwise, $e$ stays closed. In other words, an open cluster stops growing as soon as it becomes infinite: it freezes, hence the name \emph{frozen
percolation} for this process.

The above description is informal -- it is not clear that such a process exists. In \cite{Aldous2000}, Aldous studies the special cases where
$G$ is the infinite binary tree (where every vertex has degree three), or the planted binary tree (where one vertex, the root vertex, has degree one, and all other vertices have degree three).
He showed that the frozen percolation process exists for these choices of $G.$ However, Benjamini
and Schramm \cite{Benjamini1999} showed that for $G=\mathbb{Z}^{2},$ there is no process satisfying
the aforementioned evolution. For more details see Remark (i) after Theorem 1 of \cite{Berg2001}.
It seems that no simple condition on the graph
$G$ is known that guarantees the existence of the frozen percolation process.

To get more insight in the non-existence for $\mathbb{Z}^2,$ a modification of the process was studied in \cite{Berg}.
In the modified process, an open cluster freezes as soon as it reaches size at least $N,$ where $N$ (a positive integer)
is the parameter of the model. See Definition \ref{def: size measurement} below for the meaning of `size`. 
Formally, the evolution of a frozen percolation process with parameter $N$ is the following.

At time $0$, every edge is closed. At time $t$, an edge $e=(u,v)\in E$
becomes open if $\tau_{(u,v)}=t$ \emph{and} the open clusters of
$u$ and $v$ at time $t$ have size strictly smaller than $N$ -- otherwise, $e$ stays closed. We call this modified process the $N$-parameter frozen percolation process.
Note that replacing $N$ by $\infty$ corresponds formally to Aldous' infinite frozen percolation process, 
therefore we sometimes refer to it as the $\infty$-parameter frozen percolation process.

The $N$-parameter frozen percolation process does exist on $\mathbb{Z}^2$ (and on many other other graphs including the binary tree), 
since it can be described as a finite-range interacting particle system. For general existence results of interacting particle systems,
see for example Chapter 1 of \cite{Liggett2005}.
Van den Berg, de Lima and Nolin \cite{Berg} study the distribution of the final cluster size (i.e. the size of the cluster of a given vertex at time $1$). 
They show that, for $\mathbb{Z}^2,$ the final cluster size
is smaller than $N$, but still of order of $N$, with probability bounded away from $0$. In the light of the earlier mentioned fundamental difference (the 
existence versus the non-existence of the $\infty$-parameter frozen percolation process),
it is natural to ask if the $N$-parameter process for the planted binary behaves, for large $N,$ very differently from that on $\mathbb{Z}^2.$
It turns out that this is indeed the case: We show that the $N$-parameter 
frozen percolation process for the planted binary tree converges (in some sense, see Theorem \ref{thm:volume freeze main}) to Aldous' process as the parameter goes to infinity. 
In particular, the probability that the final cluster has size less than $N,$ but of order $N,$ converges to $0$ (see \eqref{eq:bignotfrozen prob} below).

Before stating our main result, let us give some notation. We distinguish between different frozen percolation processes by using subscripts for the probability measures.
We thus use $\mathbb{P}_{N}$ to denote the probability measure for the $N$-parameter frozen percolation process where the size of a cluster is measured by its volume, while for the $\infty$-parameter frozen percolation process, we use the notation $\mathbb{P}_{\infty}.$
We denote the open cluster of the root vertex at time $t$ by $\mathcal{C}_{t}.$ For a connected sub-graph (cluster) $C$ of the graph $G,$ the volume
of $C,$ i.e. the number of edges of $C,$ will be denoted by $|C|.$ Our main result is the following.
\begin{thm}
\label{thm:volume freeze main} For the $N$-parameter frozen percolation process on the planted binary tree, where the size of a cluster is measured by its volume, we have
  \[\mathbb{P}_{N}\left(\mathcal{C}_{t}=C\right)\rightarrow\mathbb{P}_{\infty}\left(\mathcal{C}_{t}=C\right)\mbox{ as }N\rightarrow\infty\]
  for all finite clusters $C.$ Moreover 
  \begin{equation}
    \lim_{k\rightarrow\infty}\limsup_{N\rightarrow\infty}\mathbb{P}_{N}\left(k\leq\left|\mathcal{C}_{t}\right|<N\right)=0,\label{eq:bignotfrozen prob}
  \end{equation}
  and hence the probability that the open cluster of the root vertex is frozen also converges:\[
  \mathbb{P}_{N}\left(N\leq\left|\mathcal{C}_{t}\right|\right)\rightarrow\mathbb{P}_{\infty}\left(\left|\mathcal{C}_{t}\right|=\infty\right)\mbox{ as }N\rightarrow\infty.\]
\end{thm}

The theorem above considers the case where size is measured by the volume. It can be extended to other notions of size.
To state our more general result, we need to introduce some additional definitions. We denote the planted binary tree by $T,$ and by $\mathscr{C}$ the set of finite clusters (finite connected components) of $T.$ 

\begin{defn}\label{def:hom}
  We say that a function $h$ on the set of vertices of $T$ into itself is a \emph{homomorphism} if it maps any edge $(s,t)$, with $s$ closer to the root than $t$, to an edge $(h(s),h(t))$, with $h(s)$ closer to the root than $h(t)$.
\end{defn}

\begin{defn}
  \label{def: size measurement}A \emph{good} size function of clusters is a
  function $s:\mathscr{C}\rightarrow\mathbb{N},$ which satisfies the following conditions:
  \begin{enumerate}
    \item \label{it: trans inv} \textit{Compatibility with homomorphisms.} For all $C\in\mathscr{C}$ and injective homomorphisms $h$ we have $s(h(C))=s(C).$
    \item \label{it: goes to inf} \textit{Finiteness.} For all $N\in\mathbb{N}$ and for any vertex $v,$ the set $\{C\in \mathscr{C} \, \left| \, v\in C,\,s(C)\leq N\right.\}$ is finite.
    \item \label{it: inc size} \textit{Monotonicity.} If $C,C'\in\mathscr{C}$ with $C\subseteq C',$ then $s\left(C\right)\leq s\left(C'\right).$   
    \item \label{it: bounded by vol} \textit{Boundedness above by the volume.} For all $C\in\mathscr{C},$ we have $s(C)\leq|C|.$
  \end{enumerate}
\end{defn}
The conditions of Definition \ref{def: size measurement} are satisfied for most of the usual
size functions such as the diameter (the length of the longest self-avoiding path in the cluster) or the depth 
(the length of the longest self-avoiding path starting from the root).

We indicate the dependence on the size function with an additional superscript: $\mathbb{P}_{N}^{(s)}$ 
denotes the probability measure for the $N$-parameter frozen percolation process with size function $s.$
With this notation, the following generalization of Theorem \ref{thm:volume freeze main} holds.
\begin{thm}
  \label{thm:gen size freeze main} Let $s$ be a good size function for the planted binary tree. Then we have
  \begin{equation}
    \mathbb{P}_{N}^{(s)}\left(\mathcal{C}_{t}=C\right)\rightarrow\mathbb{P}_{\infty}\left(\mathcal{C}_{t}=C\right)\mbox{ as }N\rightarrow\infty \label{eq:finite cluster prob conv}
  \end{equation}
    for all finite clusters $C.$ Moreover
  \begin{equation}
    \lim_{k\rightarrow\infty}\limsup_{N\rightarrow\infty}\mathbb{P}_{N}^{(s)}\left(k\leq s\left(\mathcal{C}_{t}\right)<N\right)=0, \label{eq:gen main - 2}
  \end{equation}
  and hence the probability that the open cluster of the root vertex is frozen also converges:\[
  \mathbb{P}_{N}^{(s)}\left(N\leq s\left(\mathcal{C}_{t}\right) \right)\rightarrow\mathbb{P}_{\infty}\left(\left|\mathcal{C}_{t}\right|=\infty\right).\]
\end{thm}
\begin{rem}
  Equation \eqref{eq:finite cluster prob conv} is valid even without condition \ref{it: bounded by vol}  of Definition \ref{def: size measurement}.
\end{rem}

\begin{rem}
  The behaviour described in Theorem \ref{thm:gen size freeze main} is very different from that of the square lattice:
  In \cite{Berg} it is showed that for $G=\mathbb{Z}^{2},$ and for any fixed $a,b\in\mathbb{R}$ with $0<a<b<1$
  \begin{equation}
    \liminf_{N\rightarrow\infty}\mathbb{P}_{N}^{(diam)}\left(aN<diam\left(\mathcal{C}_{t}\right)<bN\right)>0, \label{eq:res Rob}
  \end{equation}
  where $diam$ denotes the diameter, while this probability tends to $0$ when $G$ is the planted binary tree, thanks to Eq.\eqref{eq:gen main - 2}. \end{rem}

Let us finally mention that since Aldous' seminal paper \cite{Aldous2000}, several related questions were studied. For example, Chapter 4 of \cite{Brouwer2005} considers frozen percolation on $\mathbb{Z}$, and variants of that model are investigated in \cite{Rath2009} and \cite{Bertoin2010}, respectively on the complete graph and on the binary tree.

The paper is organized as follows. In Section \ref{sec: volume freeze}
we prove Theorem \ref{thm:volume freeze main}. The proof relies on a careful study of the probability that the root edge is closed at time $t,$
which we denote by $\beta_{N}(t).$ In Sections \ref{sub:setting} and
\ref{sub:differential equation for beta} we show that $\beta_{N}$
satisfies a first order differential equation which involves the generating function of the Catalan numbers. In Section \ref{sub:sol ode for beta},
we give an implicit solution of the aforementioned differential equation, and we use this in Sections
\ref{sub:bound beta} and \ref{sub:conv beta} to prove the convergence
of $\beta_{N}$ as $N\rightarrow\infty.$ We finish the proof of Theorem
\ref{thm:volume freeze main} in Section \ref{sub:end proof main thm volume}. In Section \ref{sec:other size}
we point out the changes in the proof of Theorem \ref{thm:volume freeze main} required
to prove Theorem \ref{thm:gen size freeze main}.

\section{\label{sec: volume freeze} Proof of Theorem \ref{thm:volume freeze main}}

\subsection{\label{sub:setting}Setting}

In this section, we consider the $N$-parameter frozen percolation process where the size of a cluster is measured by its number of edges -- we recall the notation $\mathbb{P}_N.$ 
We denote by $\mathcal{A}_t$ the set of open edges at time $t.$

Let $e_{0}=(v_{0},v_{1})$ be the root edge, where $v_{0}$ is the
root vertex. The central quantity of our analysis is the following
probability:
\begin{equation}
  \beta_{N}\left(t\right):=\mathbb{P}_{N}\left(e_{0}\notin\mathcal{A}_{t}\right)=\mathbb{P}_{N}\left(e_{0}\mbox{ is closed at time }t\right)\label{eq:def beta}
\end{equation}
(note that $\beta_{N}(t)=\mathbb{P}_{N}\left(\left|\mathcal{C}_{t}\right|=0\right)$).

\begin{rem}
\label{rem:lbound beta prime}
From the definition, it is easy to see that $\beta_{N}\left(t\right)$
is decreasing in $t.$ Moreover, from the equality 
\begin{equation}
  \beta_{N}\left(t\right)=1-t+\mathbb{P}_{N}\left(\tau_{e_{0}}<t\mbox{ but }e_{0}\mbox{ is closed at time }t\right), \label{eq:prop beta}
\end{equation}
we can see that $(\beta_{N}\left(t\right) - 1 + t)$ is increasing in $t$.
\end{rem}

For $e\in E,$ $e\neq e_{0},$ $T\setminus\left\{ e\right\} $ has
two connected components, one which contains $e_{0},$ and one which
does not. Let $T_{e}$ denote the component which does not contain $e_{0}$, together
with the edge $e$: $T_{e}$ is a subtree of $T$, isomorphic
to $T$.

For any edge $e_{1},$ we define the frozen percolation process on
$T_{e_{1}}$ in the following way. We consider the set of random variables
$\tau_{e},$ $e\in T_{e_{1}},$ and define the frozen percolation
process on $T_{e_{1}}$ in the same way as we did for $T.$ We denote the set of open edges at time $t$ by $\mathcal{A}_{t}\left(e_{1}\right).$
Note that the process $\mathcal{A}_{t}\left(e_{1}\right)$ has the
same law as $\mathcal{A}_{t}.$ Moreover, $\mathcal{A}_{t}\left(e_{1}\right)$
and $\mathcal{A}_{t}$ are coupled via the random variables $\tau_{e},$
$e\in T_{e_{1}}.$

In the following, we think of clusters as sets of edges. The outer boundary
of a cluster $C\subseteq E$, denoted by $\partial C$, is the set of edges in $E\setminus C$ that have a common endpoint with
one of the edges of $C$.

\subsection{\label{sub:differential equation for beta}Differential equation
for $\beta_{N}$}

Let us denote the $k$th Catalan number by $c_{k}=\binom{2k}{k}/\left(k+1\right),$ and recall that the generating function of the Catalan numbers is (see for example Section 2.1 of \cite{Drmota2009})\[
C\left(x\right)=\sum_{k=0}^{\infty}c_{k}x^{k}=\frac{1-\sqrt{1-4x}}{2x}=\frac{2}{1+\sqrt{1-4x}},\]
which converges for $\left|x\right| \leq \frac{1}{4}$. If we denote by $C_N$ the $N$th partial sum, that is
$$C_{N}\left(x\right)=\sum_{k=0}^{N}c_{k}x^{k},$$
we have:
\begin{lem}
  \label{lem:diff beta}$\beta_{N}$ is differentiable, and its derivative satisfies
  \begin{equation}
    \beta_{N}'\left(t\right)=-\frac{\beta_{N}\left(t\right)}{t}\big[C_{N}\left(t\beta_{N}\left(t\right)\right)-1\big].\label{eq:diff beta}
  \end{equation}
\end{lem}

\begin{rem}\label{rem:ode beta}
Since $C_{N}\left(x\right)=1+x+\ldots,$ Eq.\eqref{eq:diff beta}
is well defined for $t=0.$ In the
introduction we pointed out that the model exists, in particular the differential
equation \eqref{eq:diff beta} with initial condition $\beta_{N}\left(0\right)=1$
has a solution. On the other hand, the general theory of ordinary differential equations
provides uniqueness.
\end{rem}
\begin{proof}
Let us denote the open cluster of $v_{1}$ without the edge $e_{0}$
at time $s$ by $\tilde{\mathcal{C}}_{s}$.

We use the defining evolution of the $N$-parameter frozen percolation process as follows: At time $s,$ if $\tau_{e_{0}}=s,$ then
$e_{0}$ tries to become open, and it succeeds if and only if $\left|\tilde{\mathcal{C}}_{s}\right|\leq N-1.$
By conditioning on $\tau_{e_{0}},$ we get that 
\begin{align}
  \beta_{N}\left(t\right) & =1-\int_{0}^{t}\mathbb{P}_{N}\left(\left|\tilde{\mathcal{C}}_{s}\right|<N|\tau_{e_{0}}=s\right)ds\nonumber \\
  & =1-\int_{0}^{t}\sum_{k=0}^{N-1}\mathbb{P}_{N}\left(\left|\tilde{\mathcal{C}}_{s}\right|=k\,|\,\tau_{e_{0}}=s\right)ds.\label{eq:beta 1-int}
\end{align}
First we compute the probability $\mathbb{P}_{N}\left(\tilde{\mathcal{C}}_{s}=C\,|\,\tau_{e_{0}}=s\right)$
for $\left|C\right|\leq N-1.$ If $\tilde{\mathcal{C}}_{s}=C,$ $\left|C\right|\leq N-1,$
then for all $e\in C,$ $e$ is open at time $s.$ Moreover, for all
$e'\in\partial C\setminus\left\{ e_{0}\right\} ,$ $e'$ is closed at time $s.$
The latter event can happen in two ways: $e'$ is closed at time $s$ in its own
frozen percolation process on $T_{e'},$ or there is a big cluster
at time $s$ in $T\setminus T_{e'}$ touching $e'.$ Since $\left|C\right|<N,$
on the event $\left\{ \tilde{\mathcal{C}}_{s}=C,\tau_{e_{0}}=s\right\} ,$ the latter
cannot happen. Hence 
\[\left\{ \tilde{\mathcal{C}}_{s}=C,\tau_{e_{0}}=s\right\} \subseteq\bigcap_{e'\in\partial C\setminus\left\{ e_{0}\right\} }\left\{ e'\notin\mathcal{A}_{s}\left(e'\right)\right\} =:A.\]
 Note that the event $A$ and the random variables $\tau_{e}$, $e\in C$
are independent. Moreover, conditionally on $A,$ the events $e\in\mathcal{A}_{s},\, e\in C$
are independent, and each of them has probability $s$, so that 
\begin{equation}
  \mathbb{P}_{N}\left(\left.\tilde{\mathcal{C}}_{s}=C\,\right|e\,'\notin\mathcal{A}_{s}\left(e'\right)\mbox{ for }e'\in\partial C\setminus\left\{ e_{0}\right\} ,\tau_{e_{0}}=s\right)=s^{\left|C\right|}.\label{eq:int form beta 1}
\end{equation}
Recall that the processes $\mathcal{A}_{s}\left(e'\right),$ $e'\in\partial C\setminus\left\{ e_{0}\right\}$
are independent and have the same law as $\mathcal{A}_{s}.$ Hence
the events $e'\notin\mathcal{A}_{s}\left(e'\right),e'\in\partial C\setminus\left\{ e_{0}\right\} $
are independent, and each of them has probability $\beta_{N}\left(s\right).$
This together with \eqref{eq:int form beta 1} gives that 
\[\mathbb{P}_{N}\left(\tilde{\mathcal{C}}_{s}=C\,|\,\tau_{e_{0}}=s\right)=s^{\left|C\right|}\beta_{N}\left(s\right)^{\left|\partial C\setminus\left\{ e_{0}\right\} \right|}.\]
Using that $\left|\partial \tilde{\mathcal{C}}_{s}\setminus\left\{ e_{0}\right\} \right|=\left|\tilde{\mathcal{C}}_{s}\right|+2$,
we get
\begin{equation}
  \mathbb{P}_{N}\left(\tilde{\mathcal{C}}_{s}=C\,|\,\tau_{e_{0}}=s\right)=\beta_{N}\left(s\right)^{2}\left(s\beta_{N}\left(s\right)\right)^{\left|C\right|}. \label{eq:int form beta 2}
\end{equation}
It is well known that the number of clusters $C\subseteq T$ having $k$ edges which
contain the vertex $v_{1}$ but not the edge $e_{0}$ is $c_{k+1},$ the $(k+1)$th Catalan number (see for example Theorem 2.1 of \cite{Drmota2009}).
By this and \eqref{eq:int form beta 2} we can rewrite \eqref{eq:beta 1-int} as follows:
\begin{align}
  \beta_{N}\left(t\right) & =1-\int_{0}^{t}\beta_{N}\left(s\right)^{2}\sum_{k=0}^{N-1}c_{k+1}\left(s\beta_{N}\left(s\right)\right)^{k}ds.\nonumber \\
  & =1-\int_{0}^{t}\frac{\beta_{N}\left(s\right)}{s}\left(C_{N}\left(s\beta_{N}\left(s\right)\right)-1\right)ds.\label{eq:int beta}
\end{align}
Recall that $C_N(x)=1+x+\ldots,$ hence for every fixed positive integer $N,$ the integrand in \eqref{eq:int beta}
is bounded (since $s,\beta_{N}\left(s\right)\in\left[0,1\right]$ and $C_N$ is continuous).
Thus we can differentiate Eq.\eqref{eq:int beta}, which
completes the proof of Lemma \ref{lem:diff beta}.
\end{proof}

\subsection{\label{sub:sol ode for beta}Implicit formula for $\beta_{N}$}

Lemma \ref{lem:imp beta} gives an implicit solution of \eqref{eq:diff beta} with initial condition $\beta_{N}\left(0\right)=1.$
Before stating and proving the proposition, let us give a heuristic computation to explain where that proposition comes from,
without checking if the operations performed are legal or not.

Define the function $\gamma_{N}\left(t\right)=t\beta_{N}\left(t\right).$ It follows from Eq.\eqref{eq:diff beta} that $\gamma_{N}$ satisfies 
\[\frac{\gamma_{N}'\left(t\right)}{\gamma_{N}\left(t\right)\left(2-C_{N}\left(\gamma_{N}\left(t\right)\right)\right)}=\frac{1}{t},\]
so
\[\int_{a}^{\gamma_{N}\left(t\right)}\frac{dx}{x\left(2-C_{N}\left(x\right)\right)}=\log t+b\]
for some constants $a,b$. Using $\int_{a}^{\gamma_{N}\left(t\right)}\frac{dx}{x}=\log t+\log\left(\beta_{N}\left(t\right)/a\right),$ we get
\begin{equation}
  \int_{a}^{\gamma_{N}\left(t\right)}\frac{C_{N}\left(x\right)-1}{x\left(2-C_{N}\left(x\right)\right)}dx=-\log\beta_{N}\left(t\right)+b'\label{eq:int formula gamma}
\end{equation}
for another constant $b'$. Finally, by plugging in $\beta_{N}\left(0\right)=1$ and $\gamma_{N}\left(0\right)=0$, we can evaluate $b'$, which gives 
\[\int_{0}^{t\beta_{N}\left(t\right)}\frac{C_{N}\left(x\right)-1}{x\left(2-C_{N}\left(x\right)\right)}dx=-\log\beta_{N}\left(t\right).\]
This suggests the following lemma.

\begin{lem}
\label{lem:imp beta} For $t\in\left[0,1\right]$, $\beta_{N}\left(t\right)$
is the unique positive solution of the equation in $z$
\begin{equation}
\int_{0}^{tz}\frac{C_{N}\left(x\right)-1}{x\left(2-C_{N}\left(x\right)\right)}dx+\log z=0,\label{eq:implicit beta}\end{equation}
 with the constraint $tz<x_{N}$, where $x_{N}$
is the unique positive solution of $C_{N}\left(x\right)-2=0.$ 
\end{lem}

\begin{proof} Let us fix $N.$ First, the polynomial $C_{N}\left(x\right)-2$ has a positive derivative for $x>0$, it has thus exactly one non-negative root $x_N$, and this root has multiplicity one. Note that $x_{N}>1/4$, since $C(x)>C_{N}\left(x\right)$
for $x\in\left(0,1/4\right],$ and $C\left(1/4\right)=2.$ ($C_{N}\left(x\right)$
and $C\left(x\right)$ are close for large $N,$ this also suggests
that the root is close to $1/4$ for large $N$: we will indeed prove
that in the following.)

Let us prove that for $t\in\left[0,1\right],$ there
is exactly one non-negative solution of \eqref{eq:implicit beta}
with $tz<x_{N}.$ The integrand
in \eqref{eq:implicit beta} is positive, and it is well defined at $0$ since $C_{N}\left(x\right)=1+x+O\left(x^{2}\right)$ as $x \to 0$ ($N\geq1$). As $x \nearrow x_{N}$, this integrand behaves like $\frac{\kappa}{x_{N}-x}$ for some positive constant $\kappa$ (using that the positive root $x_N$ of $C_{N}\left(x\right)-2$ has multiplicity one). Hence, 
\begin{equation}
  \int_{0}^{x_{N}}\frac{C_{N}\left(x\right)-1}{x\left(2-C_{N}\left(x\right)\right)}dx=\infty.\label{eq:int till xn}
\end{equation}
On the other hand,
\[\int_{0}^{z}\frac{C_{N}\left(x\right)-1}{x\left(2-C_{N}\left(x\right)\right)}dx<\infty\]
for $z\in[0,x_{N})$. This shows that for every $t\in\left[0,1\right]$,
there is exactly one positive real number $u_N\left(t\right)$
which satisfies the equation \eqref{eq:implicit beta}, and $tu_N\left(t\right)<x_{N}.$

To complete the proof of Lemma \ref{lem:imp beta}, it is enough to show that $u_N$ is differentiable, that
\begin{equation}
  u_N'(t)=-\frac{u_N(t)}{t}\big[C_N(tu_N(t))-1\big] \label{eq:ode}
\end{equation}
for $t\in[0,1]$, and that $u_N(0)=1.$ Indeed, as already noted in Remark \ref{rem:ode beta}, the differential equation \eqref{eq:ode} has a unique solution.
A substitution into \eqref{eq:implicit beta} shows that $u_N\left(0\right)=1$. It is easy to check the conditions
of the implicit function theorem, and get that $u_N\left(t\right)$
is a differentiable function with derivative satisfying
\[\left(tu_N'\left(t\right)+u_N\left(t\right)\right)\frac{C_{N}\left(tu_{N}\left(t\right)\right)-1}{tu_{N}\left(t\right)\left(2-C_{N}\left(tu_{N}\left(t\right)\right)\right)}=-\frac{u_{N}'\left(t\right)}{u_{N}\left(t\right)},\]
from which simple computations give \eqref{eq:ode}. This completes the proof of Lemma \ref{lem:imp beta}.
\end{proof}


\subsection{\label{sub:bound beta}Bounds on $\beta_{N}$}

We now compare $\beta_{N}$ with the corresponding function in Aldous'
paper \cite{Aldous2000}, where clusters are frozen as soon as
they become infinite. In Aldous' model, one has
\[\beta_{\infty}\left(t\right):=\mathbb{P}_{\infty}\left(e_{0}\mbox{ is closed at time }t\right)=\begin{cases}
1-t & \text{ if } t\in\left[0,1/2\right],\\
\frac{1}{4t} & \text{ if } t\in[1/2,1].\end{cases}\]
 The following bounds hold true: 
\begin{lem}
\label{thm:bound beta}We have
$$0\leq\beta_{N}\left(t\right)-\beta_{\infty}\left(t\right)\leq2\left(x_{N}- 1/4\right) \quad \text{for all $t\in\left[0,1\right]$},$$
where $x_{N}$ ($>1/4$) is the unique positive root of the polynomial $C_{N}\left(x\right)-2.$ 
\end{lem}
\begin{proof}
From Lemma \ref{lem:imp beta}, we know that $t\beta_{N}\left(t\right)<x_{N},$
which gives the desired upper bound for $t\in\left[\frac{1}{2},1\right].$ We also know (Remark \ref{rem:lbound beta prime}) that $\beta_{N}\left(t\right)-1+t$ is non-negative and increasing. Hence,
\begin{equation}
  0\leq\beta_{N}\left(t\right)-1+t\leq\beta_{N}\left(1/2\right)-1/2\leq2\left(x_{N}-1/4\right)\label{eq:bound beta 1}
\end{equation}
for $t\in\left[0,\frac{1}{2}\right]$, by using also the previously proven upper
bound at $t=\frac{1}{2}$. We have thus established the desired lower and upper bounds for $t\in\left[0,\frac{1}{2}\right]$. In particular, for $t=\frac{1}{2},$ we obtain that $\beta_N(1/2) \geq 1/2$.

Now, let us note that $t\beta_N(t)$ is increasing: this is an easy consequence of two facts, that $\beta_{N}(t)$ is decreasing and that the integrand in the left hand-side of \eqref{eq:implicit beta} is positive. Combined with the bound $\beta_N(1/2) \geq 1/2$, we get
\[\frac{1}{4}\leq\frac{1}{2}\beta_N(1/2)\leq t\beta_N(t),\] from which the desired lower bound for  $t\in\left[\frac{1}{2},1\right]$ follows readily.
This completes the proof of Lemma \ref{thm:bound beta}.
\end{proof}


\subsection{\label{sub:conv beta}Convergence to $\beta_{\infty}$}

It follows from Lemma \ref{thm:bound beta} that in order to prove uniform convergence of the functions
$\beta_{N}$ to $\beta_{\infty}$, it is enough to prove that $x_{N}\rightarrow1/4$
as $N\rightarrow\infty.$ We prove a bit more, namely we give an upper bound on the rate of convergence.
\begin{prop}
\label{pro:bound xn}There exists a constant $K$ such that $x_{N}-\frac{1}{4}<\frac{K}{N}.$ In particular,
$$0\leq\beta_{N}\left(t\right)-\beta_{\infty}\left(t\right)\leq \frac{2 K}{N} \quad \text{for all $t \in [0,1]$},$$
so that $\beta_N \to \beta_{\infty}$ uniformly on $[0,1]$.
\end{prop}
Proposition \ref{pro:bound xn} follows from the following lemma. 
\begin{lem}
\label{lem:conv Cn}The functions $\sqrt{N}\left(C_{N}\left(\frac{1}{4}+\frac{x}{4N}\right)-2\right)$
converge locally uniformly in $x\in\mathbb{R}$ as $N\rightarrow\infty$
to the function\[
F\left(x\right)=\frac{2}{\sqrt{\pi}}\left(\sqrt{x}\int_{0}^{x}\frac{e^{y}}{\sqrt{y}}dy-e^{x}\right).\]
 
\end{lem}
\begin{proof}[Proof of Proposition \ref{pro:bound xn}] Let us take
$K\in\mathbb{R},\, K>0$ such that $F\left(K\right)>1$ (such a $K$
exists, since $F\left(x\right)\sim\frac{1}{\sqrt{\pi}}\frac{e^{x}}{x}\rightarrow\infty$
as $x\rightarrow\infty$). Then by Lemma \ref{lem:conv Cn}, we have
that for large $N$, 
\[\sqrt{N}\left(C_{N}\left(\frac{1}{4}+\frac{K}{4N}\right)-2\right)\geq F\left(K\right)-\frac{1}{2}>1-\frac{1}{2}>\frac{1}{2},\]
and so
\[C_{N}\left(\frac{1}{4}+\frac{K}{4N}\right)-2>0.\]
For any fixed $N,$ the function $x\mapsto C_{N}\left(\frac{1}{4}+\frac{x}{4N}\right)-2$
is increasing on $[0,\infty).$ Hence, $\frac{1}{4}+\frac{K}{4N}>x_{N},$
that is $x_{N}-\frac{1}{4} < \frac{K}{4N}.$
\end{proof}

\begin{proof}[Proof of Lemma \ref{lem:conv Cn}] Using that
\[2 = C(1/4) = \sum_{k=0}^{\infty}\frac{\binom{2k}{k}}{k+1}4^{-k},\]
we get
\begin{align}
\sqrt{N}\bigg(C_{N} & \bigg(\frac{1}{4}+\frac{x}{4N}\bigg)-2\bigg)\nonumber \\
& = \sqrt{N}\sum_{k=0}^{N} \frac{\binom{2k}{k}}{k+1} 4^{-k} \left(\left(1+x/N\right)^{k}-1\right) - \sqrt{N} \sum_{k=N+1}^{\infty}\frac{\binom{2k}{k}}{k+1}4^{-k}\nonumber \\
& =: (A) - (B).\label{eq:lem conv Cn - 1}
\end{align}
We will use the following version of Stirling's formula:
\begin{equation} \label{stirling}
k!=\sqrt{2\pi k}\left(\frac{k}{e}\right)^{k}e^{\lambda_{k}} \quad \text{with $\frac{1}{12k+1}<\lambda_{k}<\frac{1}{12k}$}.
\end{equation}

Using this formula, we obtain that
\begin{equation*}
(B) = \frac{1}{\sqrt{\pi}}\sum_{k=N+1}^{\infty} \frac{\sqrt{N}}{\sqrt{k}\left(k+1\right)} e^{\lambda_{2k}-2\lambda_{k}} = \frac{1}{\sqrt{\pi}}\frac{1}{N}\sum_{k=N+1}^{\infty}\frac{1}{\sqrt{\frac{k}{N}}\frac{k+1}{N}}e^{\lambda_{2k}-2\lambda_{k}},
\end{equation*}
and thus
\begin{align}
(B) & =\frac{1}{\sqrt{\pi}}\left(1+O\left(N^{-1}\right)\right)\frac{1}{N}\sum_{k=N+1}^{\infty}\frac{1}{\sqrt{\frac{k}{N}}\frac{k}{N}}\nonumber \\
& =\frac{1}{\sqrt{\pi}}\left(1+O\left(N^{-1}\right)\right)\left(\int_{1}^{\infty}y^{-3/2}dy+O\left(N^{-3/2}\right)\right) = \frac{2}{\sqrt{\pi}}+O\left(N^{-1}\right).\label{eq:lem conv Cn - 1.1}
\end{align}


We now divide (A) into two parts. On the one hand, using Eq.\eqref{stirling}, we get that for some universal constants
$C$, $C'$,
\begin{align}
  \sqrt{N}\Bigg|\sum_{k=0}^{\lfloor\sqrt{N}\rfloor}\frac{\binom{2k}{k}}{k+1} & 4^{-k}\left(\left(1+x/N\right)^{k}-1\right)\Bigg|\nonumber \\
  & \leq\sqrt{N}C\sum_{k=1}^{\lfloor\sqrt{N}\rfloor}\frac{1}{\sqrt{k}\left(k+1\right)}\left|\left(1+\frac{x}{N}\right)^{k}-1\right|\nonumber \\
  & \leq\frac{\left|x\right|}{\sqrt{N}}C\sum_{k=1}^{\lfloor\sqrt{N}\rfloor}\frac{1}{\sqrt{k}\left(k+1\right)}\left(1+\left(1+\frac{\left|x\right|}{N}\right)+\ldots+\left(1+\frac{\left|x\right|}{N}\right)^{k-1}\right)\nonumber \\
  & \leq\frac{\left|x\right|}{\sqrt{N}}C\sum_{k=1}^{\lfloor\sqrt{N}\rfloor}\frac{1}{\sqrt{k}}\left(1+\frac{\left|x\right|}{N}\right)^{k-1}\nonumber \\
  & \leq\frac{\left|x\right|}{\sqrt{N}}C\left(1+\frac{\left|x\right|}{N}\right)^{\sqrt{N}}\sum_{k=1}^{\lfloor\sqrt{N}\rfloor}\frac{1}{\sqrt{k}}\nonumber \\
  & \leq C'\left|x\right|e^{\left|x\right|}N^{-1/4}.\label{eq:lem conv Cn - 1.2}
\end{align}

%

On the other hand, using again Eq.\eqref{stirling},
\begin{align}
  \sqrt{N} \sum_{k=\lfloor\sqrt{N}\rfloor+1}^{N}\frac{\binom{2k}{k}}{k+1} & 4^{-k}\left(\left(1+x/N\right)^{k}-1\right)\nonumber \\
  & =\sqrt{\frac{N}{\pi}}\sum_{k=\lfloor\sqrt{N}\rfloor+1}^{N}\frac{1}{\sqrt{k}\left(k+1\right)}e^{\lambda_{2k}-2\lambda_{k}}\left(\left(1+\frac{x}{N}\right)^{k}-1\right)\nonumber \\
  & =\frac{1}{\sqrt{\pi}}\left(1+O\left(N^{-1/2}\right)\right)\frac{1}{N}\sum_{k=\lfloor\sqrt{N}\rfloor+1}^{N}\frac{1}{\sqrt{\frac{k}{N}}\frac{k+1}{N}}\left(\left(1+\frac{x}{N}\right)^{N\left(k/N\right)}-1\right)\nonumber \\
  & =\frac{1}{\sqrt{\pi}}\left(1+O\left(N^{-1/2}\right)\right)\left(\int_{1/\sqrt{N}}^{1}y^{-3/2}\left(e^{xy}-1\right)dy+e^{\left|x\right|}O\left(N^{-3/2}\right)\right)\nonumber \\
  & =\frac{1}{\sqrt{\pi}}\int_{0}^{1}y^{-3/2}\left(e^{xy}-1\right)dy+e^{\left|x\right|}O\left(N^{-1/4}\right).\label{eq:lem conv Cn - 1.3}
\end{align}

Substituting \eqref{eq:lem conv Cn - 1.1}, \eqref{eq:lem conv Cn - 1.2}
and \eqref{eq:lem conv Cn - 1.3} into \eqref{eq:lem conv Cn - 1},
we get
\begin{equation} \sqrt{N}\bigg(C_{N}\bigg( \frac{1}{4}+\frac{x}{4N}\bigg)-2\bigg) =\frac{1}{\sqrt{\pi}}\int_{0}^{1}y^{-3/2}\left(e^{xy}-1\right)dy-\frac{2}{\sqrt{\pi}}+\left(1+\left|x\right|\right)e^{\left|x\right|}O\left(N^{-1/4}\right).\label{eq:lem conv Cn - 2}
\end{equation}

Finally, an integration by parts gives 
\begin{align}
  \int_{0}^{1}y^{-3/2}\left(e^{xy}-1\right)dy & =\left[\frac{y^{-1/2}}{-1/2}\left(e^{xy}-1\right)\right]_{y=0}^{y=1}-\int_{0}^{1}\frac{y^{-1/2}}{-1/2}xe^{xy}dy\nonumber \\
  & =-2\left(e^{x}-1\right)+2x\int_{0}^{1}\frac{e^{xy}}{\sqrt{y}}dy,\label{eq:lem conv Cn - 3}
\end{align}
 and combining Eqs.\eqref{eq:lem conv Cn - 2} and \eqref{eq:lem conv Cn - 3}
(and a change of variable) completes the proof of Lemma \ref{lem:conv Cn}.
\end{proof}


\subsection{\label{sub:end proof main thm volume} Completion of the proof of Theorem \ref{thm:volume freeze main}}

Recall the notation $\mathcal{C}_{t}.$ Let $\left|C\right|<N$ be
a fixed cluster of the root vertex. By similar arguments as in the proof of Lemma \ref{lem:diff beta}, we have
\begin{equation}
  \mathbb{P}_{N}\left(\mathcal{C}_{t}=C\right)=t^{\left|C\right|}\beta_{N}\left(t\right)^{\left|\partial C\right|}=\beta_{N}\left(t\right)\left(t\beta_{N}\left(t\right)\right)^{\left|C\right|}.\label{eq:prob C equal C}
\end{equation}
(since $|\partial C| = |C|+1$). Hence for any fixed finite cluster $C$, we have, as $N\to\infty$,
\begin{equation}
  \mathbb{P}_{N}\left(\mathcal{C}_{t}=C\right)=\beta_{N}\left(t\right)\left(t\beta_{N}\left(t\right)\right)^{\left|C\right|}\rightarrow\beta_{\infty}\left(t\right)\left(t\beta_{\infty}\left(t\right)\right)^{\left|C\right|}=\mathbb{P}_{\infty}\left(\mathcal{C}_{t}=C\right),\label{eq:lim prob C equal C}
\end{equation}
which gives the first part of Theorem \ref{thm:volume freeze main}.

An argument similar to the beginning of the proof of Lemma \ref{lem:diff beta}
gives that \[
\mathbb{P}_{N}\left(k\leq\left|\mathcal{C}_{t}\right|<N\right)=\sum_{n=k}^{N-1}c_{n}\beta_{N}\left(t\right)\left(t\beta_{N}\left(t\right)\right)^{n}.\]
Lemma \ref{lem:imp beta} and Proposition \ref{pro:bound xn} then imply that $t\beta_{N}\left(t\right) < x_N \leq \frac{1}{4}+\frac{K'}{4N},$ hence
(using again Eq.\eqref{stirling})
\begin{align*}
  \mathbb{P}_{N}\left(k\leq\left|\mathcal{C}_{t}\right|<N\right) & \leq\beta_{N}\left(t\right)\sum_{n=k}^{N-1}\frac{\binom{2n}{n}}{n+1}\left(t\beta_{N}\left(t\right)\right)^{n}\\
  & \leq K_{1}\sum_{n=k}^{N-1}\frac{1}{\sqrt{n}}\frac{1}{n+1}\left(1+\frac{K'}{N}\right)^{n}\\
  & \leq K_{2}e^{K'}\sum_{n=k}^{\infty}\frac{1}{\sqrt{n}}\frac{1}{n+1}\\
  & \leq K_{3}\int_{k}^{\infty}\frac{dx}{x^{3/2}}=\frac{K_{4}}{\sqrt{k}}.
\end{align*}
It follows that
\begin{equation}
  \lim_{k\rightarrow\infty}\limsup_{N\rightarrow\infty}\mathbb{P}_{N}\left(k\leq\left|\mathcal{C}_{t}\right|<N\right)=0,\label{eq:lim prob big not freeze}
\end{equation}
which completes the second part of Theorem \ref{thm:volume freeze main}.

Now, using the trivial upper bound $\mathbb{P}_{N}\left(N\leq\left|\mathcal{C}_{t}\right|\right)\leq\mathbb{P}_{N}\left(k\leq\left|\mathcal{C}_{t}\right|\right)$ for $k\leq N$, we get
\begin{equation}	
  \limsup_{N\rightarrow\infty}\mathbb{P}_{N}\left(N\leq\left|\mathcal{C}_{t}\right|\right)\leq \lim_{k\rightarrow\infty} \limsup_{N\rightarrow\infty}\mathbb{P}_{N}\left(k\leq\left|\mathcal{C}_{t}\right|\right)= \lim_{k\rightarrow\infty} \mathbb{P}_{\infty}\left(k\leq\left|\mathcal{C}_{t}\right|\right) = \mathbb{P}_{\infty}\left(\left|\mathcal{C}_{t}\right|=\infty\right), \label{eq:limsup ubound - 2}
\end{equation}
where we used \eqref{eq:lim prob C equal C} for the first equality.

On the other hand, for all $k\in\mathbb{N}$, $k\leq N$, we have
\begin{equation}
\mathbb{P}_{N}\left(N\leq\left|\mathcal{C}_{t}\right|\right)=\mathbb{P}_{N}\left(k\leq\left|\mathcal{C}_{t}\right|\right)-\mathbb{P}_{N}\left(k\leq\left|\mathcal{C}_{t}\right|<N\right). \label{endofproof}
\end{equation}
Hence, taking first the limit infimum as $N \to \infty$, and then the limit as $k \to \infty$, we get
\begin{align}
  \liminf_{N\rightarrow\infty}\mathbb{P}_{N}\left(N\leq\left|\mathcal{C}_{t}\right|\right) &  \geq \lim_{k\rightarrow\infty} \liminf_{N\rightarrow\infty}\mathbb{P}_{N}\left(k\leq\left|\mathcal{C}_{t}\right|\right)- \lim_{k\rightarrow\infty} \limsup_{N\rightarrow\infty}\mathbb{P}_{N}\left(k\leq\left|\mathcal{C}_{t}\right|<N\right) \nonumber \\
  & = \mathbb{P}_{\infty}\left(\left|\mathcal{C}_{t}\right|=\infty\right) - 0,  \label{eq:liminf lbound - 1}
\end{align}
where for the last equality we used, respectively, \eqref{eq:lim prob C equal C} -- as in \eqref{eq:limsup ubound - 2} -- and \eqref{eq:lim prob big not freeze}.


Combining \eqref{eq:limsup ubound - 2} and \eqref{eq:liminf lbound - 1} provides the final part of Theorem \ref{thm:volume freeze main}. 
\begin{flushright}
  $\Box$
\end{flushright}

\section{\label{sec:other size}Proof of Theorem \ref{thm:gen size freeze main}}

In this section we give a brief outline of the changes required to deduce Theorem \ref{thm:gen size freeze main} 
from the arguments in Section \ref{sec: volume freeze}.

First, for any good size function $s$, the corresponding $N$-parameter frozen percolation process does exist. Indeed, conditions \ref{it: trans inv} and \ref{it: goes to inf} of Definition \ref{def: size measurement} ensure that the process is still a finite-range interacting particle system, and the general theory of such systems \cite{Liggett2005} provides existence, as in the case of volume.


In that previous case, the function $\frac{C_{N}\left(x\right)-1}{x}$
played an important role. It is the generating function of the
number of clusters of $v_{1}$ which do not contain the edge $e_{0}$ and have 
volume at most $N-1.$ For other good size functions $s,$ the following
generating function plays the role of $\frac{C_{N}\left(x\right)-1}{x}:$
\[G_{N}^{(s)}\left(x\right)=\sum_{k=0}^{\infty}a_{k,N-1}^{(s)}x^{k},\]
where $a_{k,N-1}^{(s)}$ denotes the number of clusters $C$ of $v_{1}$
for which $e_{0}\notin C,$ $\left|C\right|=k$ and $s(C)\leq N-1.$

Keeping this in mind, one can easily modify the proof of Theorem \ref{thm:volume freeze main}. We define the function $\beta_{N}^{(s)}:\left[0,1\right]\rightarrow\mathbb{R}$
as \[\beta_{N}^{(s)}\left(t\right):=\mathbb{P}_{N}^{(s)}\left(e_{0}\notin\mathcal{A}_{t}\right).\]
Using the conditions \ref{it: trans inv}, \ref{it: goes to inf} and \ref{it: inc size} of Definition \ref{def: size measurement}, by simple adjustments 
of the proof of Lemma \ref{lem:diff beta} we deduce that $\beta_{N}^{(s)}$ is differentiable, and that its derivative satisfies
\[(\beta_{N}^{(s)})'\left(t\right) = - \big(\beta_{N}^{(s)}\left(t\right)\big)^{2} G_{N}^{(s)} \left(t\beta_{N}^{(s)}\left(t\right)\right).\]
Moreover, it follows from the definition of $\beta_{N}^{(s)}$ that $\beta_{N}^{(s)}\left(0\right)=1.$ 

Recall that $x_N$, the unique positive root of $C_N(x)=2$, was another important quantity. Since in our present general setup
$G_{N}^{(s)}\left(x\right)$ plays the role of $\frac{C_{N}\left(x\right)-1}{x},$ the analogue of $x_N$ is the unique positive root of the equation $x G_{N}^{(s)}(x)=1,$
which we denote by $x_{N}^{(s)}$. Using the arguments of Section \ref{sub:sol ode for beta}, we deduce that for each
fixed $t,$ $\beta_{N}^{(s)}\left(t\right)$ is equal to the unique positive root of the equation in $z$ 
\[\int_{0}^{tz}\frac{G_{N}^{(s)}\left(x\right)}{1 - x G_{N}^{(s)}\left(x\right)}dx+\log z=0\] 
with the constraint $tz<x_{N}^{(s)}.$

By simple modifications of Section \ref{sub:bound beta}, we get that $0\leq\beta_{N}^{(s)}\left(t\right)-\beta_{\infty}\left(t\right)\leq2\left(x_{N}^{(s)}-\frac{1}{4}\right)$
for all $t\in\left[0,1\right],$ which is the analogue of Lemma \ref{thm:bound beta} in this general setting. By condition \ref{it: inc size} of Definition \ref{def: size measurement}, $a_{k,N-1}^{(s)}$ is an increasing function
of $N$ for each fixed $k.$ Moreover, since $s(C)$ is finite for all finite clusters $C,$ $a_{k,N-1}^{(s)}\uparrow c_{k+1}$ as $N\rightarrow\infty.$
Hence $G_{N}^{(s)}\left(x\right)\uparrow\frac{C\left(x\right)-1}{x}$ for
all $x\in\left[0,\frac{1}{4}\right],$ and $G_{N}^{(s)}\left(x\right)\uparrow\infty$ for $x>\frac{1}{4}.$ Thus $x_{N}^{(s)} \rightarrow \frac{1}{4}$ as $N\rightarrow\infty.$
By the aforementioned analogue of Lemma \ref{thm:bound beta}, we get that $\beta_{N}^{(s)}\rightarrow \beta_\infty$ point-wise.
This concludes the proof of the first part (Eq.\eqref{eq:finite cluster prob conv}) of Theorem \ref{thm:gen size freeze main}.

Note that up to now we did not use that $s$ satisfies Condition \ref{it: bounded by vol} of Definition \ref{def: size measurement}.
We use this condition to prove a rate of convergence for $x_{N}^{(s)},$ which was the key ingredient in the proof of \eqref{eq:bignotfrozen prob}. Condition \ref{it: bounded by vol} implies that $a_{k,N-1}^{(s)} \geq c_{k+1}$ for $k\leq N-1,$ hence
\[G_{N}^{(s)}\left(x\right)\geq\frac{C_{N}\left(x\right)-1}{x}\mbox{ for }x\geq0,\]
and thus $\frac{1}{4}\leq x_{N}^{(s)}\leq x_N = x_{N}^{(\left|.\right|)}$. Proposition
\ref{pro:bound xn} then implies that $0\leq x_{N}^{(s)}-\frac{1}{4}\leq\frac{K}{N},$
from which a computation similar to Section \ref{sub:end proof main thm volume}
completes the proof of Theorem \ref{thm:gen size freeze main}. 

\bibliographystyle{abbrv}
\bibliography{myreflist}

\begin{thebibliography}{1}

\bibitem{Aldous2000}
D.~Aldous.
\newblock The percolation process on a tree where infinite clusters are frozen.
\newblock {\em Mathematical Proceedings of the Cambridge Mathematical Society},
  128:465--477, 2000.

\bibitem{Benjamini1999}
I.~Benjamini and O.~Schramm.
\newblock Private communication with {D}avid {A}ldous, 1999.

\bibitem{Bertoin2010}
J.~Bertoin.
\newblock Fires on trees.
\newblock arXiv:1011.2308v2, 2010.

\bibitem{Brouwer2005}
R.~Brouwer.
\newblock {\em Percolation, forest-fires and monomer dimers (or the hunt for
  self-organized criticality)}.
\newblock PhD thesis, Vrije Universiteit, 2005.

\bibitem{Drmota2009}
M.~Drmota.
\newblock {\em Random trees}.
\newblock Springer-Verlag, 2009.

\bibitem{Liggett2005}
T.~M. Liggett.
\newblock {\em Interacting particle systems}.
\newblock Springer, 2005.

\bibitem{Rath2009}
B.~R\'{a}th.
\newblock Mean field frozen percolation.
\newblock {\em Journal of Statistical Physics}, 137:459--499, 2009.

\bibitem{Berg}
J.~van~den Berg, B.~N. de~Lima, and P.~Nolin.
\newblock A percolation process on the square lattice where large finite
  clusters are frozen.
\newblock To appear in Random Structures and Algorithms, arXiv:1006.2050, 2010.

\bibitem{Berg2001}
J.~van~den Berg and B.~T\'oth.
\newblock A signal-recovery system: asymptotic properties, and construction of
  an infinite-volume process.
\newblock {\em Stochastic Processes and their Applications}, 96:177--190, 2001.

\end{thebibliography}

\end{document}